\documentclass[reqno,onecolumn,oneside]{paper}
\usepackage[english]{babel}
\usepackage{enumerate,amsmath,amsthm,amsfonts,pifont,slashed,amssymb,ifsym,mathrsfs,graphicx,graphics,yfonts,calligra,
latexsym,txfonts%,pxfonts
}%\usepackage[colorlinks]{hyperref}
\usepackage[all]{xy}
\usepackage[sort]{cite}

\newtheorem{theorem}{Theorem}[section]
\newtheorem{proposition}[theorem]{Proposition}
\newtheorem{lemma}[theorem]{Lemma}
\newtheorem{corollary}[theorem]{Corollary}
\newtheorem{claim}{Claim}
\theoremstyle{definition}
\newtheorem{definition}[theorem]{Definition}
\newtheorem{exm}[theorem]{Example}

\theoremstyle{remark}
\newtheorem{remark}[theorem]{Remark}

\newcommand{\lcc}{\operatorname{lcc}}

\newcommand{\bydef}{\mathrel{\mathop:}=}

\newcommand{\op}{\operatorname{op}}

\newcommand{\id}{\operatorname{id}}

\newcommand{\Rad}{\operatorname{Rad}}

\newcommand{\supp}{\operatorname{supp}}

\newcommand{\BB}{\mathcal{B}\!\operatorname{oole}}

\newcommand{\Set}{\mathcal{S}\!\operatorname{et}}
\newcommand{\fwp}{\mathscr{F}}

\newcommand{\Q}{\mathbb{Q}}

\newcommand{\MV}{\mathcal{MV}}

\newcommand{\sk}{\operatorname{Sk}}

\renewcommand{\wp}{\mathscr{P}}

\newcommand{\TMV}{{}^{\operatorname{MV}}\!\mathcal{T}\!\!\operatorname{op}}
\renewcommand{\Top}{\mathcal{T}\!\!\operatorname{op}}
\newcommand{\Fuz}{\mathcal{F}\!\!\operatorname{uz}}
\newcommand{\SMV}{{}^{\operatorname{MV}}\!\!\mathcal{S}\!\operatorname{tone}}
\newcommand{\MVs}{\mathcal{MV}^{\operatorname{ss}}}

\newcommand{\MVlcc}{\mathcal{MV}^{\operatorname{lcc}}}

\newcommand{\LFuz}{{}^{\operatorname{L}}\!\mathcal{F}\!\!\operatorname{uz}}

\newcommand{\HCAFuz}{{}^{\operatorname{HC}}\!\mathcal{AF}\!\!\operatorname{uz}}
\newcommand{\LTMV}{{}^{\operatorname{LMV}}\!\mathcal{T}\!\!\operatorname{op}}
\newcommand{\CTMV}{{}^{\operatorname{CMV}}\!\mathcal{T}\!\!\operatorname{op}}
\newcommand{\HCTMV}{{}^{\operatorname{HCMV}}\!\mathcal{T}\!\!\operatorname{op}}
\newcommand{\HCTop}{{}^{\operatorname{HC}}\!\mathcal{T}\!\!\operatorname{op}}

\renewcommand{\phi}{\varphi}

\newcommand{\fcou}{^{{\rotatebox[origin=c]{180}{$\rightsquigarrow$}}}}
\newcommand{\restr}{\!\!\upharpoonright}
\newcommand{\la}{\langle}
\newcommand{\ra}{\rangle}

\newcommand{\lto}{\longrightarrow}

\newcommand{\lmapsto}{\longmapsto}

\newcommand{\ov}{\overline}

\newcommand{\0}{\mathbf{0}}
\renewcommand{\1}{\mathbf{1}}

\def\amslatex\slash{{\protect\AmS-\protect\LaTeX}}

\begin{document}

\title{Compactness in MV-Topologies: Tychonoff Theorem and Stone-\v Cech Compactification}

\author{\renewcommand{\thefootnote}{\arabic{footnote}}
\rm Luz Victoria De La Pava\footnotemark[1] \and
\renewcommand{\thefootnote}{\arabic{footnote}}
\rm Ciro Russo\footnotemark[2]}

\maketitle
\footnotetext[1]{Departamento de Matem\'aticas, Universidad del Valle -- Cali, Colombia. \\ {\tt victoria.delapava@correounivalle.edu.co}}
\footnotetext[2]{Departamento de Matem\'atica, Universidade Federal da Bahia -- Salvador, Bahia, Brazil. \\ {\tt ciro.russo@ufba.br}}
\today

\begin{abstract}
In this paper, we discuss some questions about compactness in MV-topological spaces. More precisely, we first present a Tychonoff theorem for such a class of fuzzy topological spaces and some consequence of this result, among which, for example, the existence of products in the category of Stone MV-spaces and, consequently, of coproducts in the one of limit cut complete MV-algebras. Then we show that our Tychonoff theorem is equivalent, in ZF, to the Axiom of Choice, classical Tychonoff theorem, and Lowen's analogous result for lattice-valued fuzzy topology. Last, we show an extension of the Stone-\v Cech compactification functor to the category of MV-topological spaces, and we discuss its relationship with previous works on compactification for fuzzy topological spaces.
\end{abstract}

\section{Introduction}
\label{intro}

MV-topological spaces were introduced by the second author with the aim of extending Stone duality to semisimple MV-algebras \cite{rus2}. In the same paper, several basic notions and results of general topology have been extended to such a class of fuzzy topologies. The results presented there indicate that MV-topological spaces constitute a pretty well-behaved fuzzy generalization of classical topological spaces. On the algebraic side, the class of algebras which play for MV-topologies the role that Boolean algebras play for classical topologies is the one of \emph{limit cut complete} (\emph{lcc}, for short) MV-algebras.

In this paper, we study some important questions about compactness in the category of MV-topological spaces. First of all, we recall some basic definitions and results on the topic, and we describe how MV-topologies behave with respect to some important and well-known functors between the categories of fuzzy topological spaces and classical ones. Then we characterize product MV-topologies by means of the concepts of \emph{subbase} and \emph{large subbase}.

Once all the preliminaries are settled, we present two proofs of Tychonoff theorem for MV-topologies, i.e., of the fact that compactness is preserved by product in this category (Theorem \ref{tych}). The first one follows from Lemma \ref{alex}, i.e., an analogous of Alexander Subbase Lemma, similarly to the classical case, while the other one is an easy consequence of Theorems \ref{ultra} and \ref{low333}.

In Section \ref{cons} we present several consequences of our Tychonoff theorem, among which the fact that the category of Stone MV-spaces has products and, consequently, the one of lcc MV-algebras has coproducts. We also discuss some properties of coproducts of lcc MV-algebras and the relationship of our main result with the Axiom of Choice and with Lowen's corresponding theorem for lattice-valued fuzzy topologies \cite{low2}.

In the last section, we use some categorical results proved by Cerruti in order to show the existence of a Stone-\v Cech Compactification for each MV-topological space. More precisely, using the properties of the functors $\iota$ and $\omega$ discussed in the previous sections, we show that the compactification functor $\hat\beta$, defined by Cerruti for a larger category, once restricted to MV-topologies, boasts properties that are similar to those of the Stone-\v Cech compactification of classical topological spaces. Moreover, we prove that the compactification of each topologically generated MV-space is completely determined by the Stone-\v Cech compactification of its initial topology.

For any notion or result on MV-topology not explicitly reported here we refer the reader to \cite{rus2}, while standard references for MV-algebras are \cite{mvbook} and \cite{mubook}. We also refer the reader to the comprehensive overviews of earlier works on fuzzy topologies presented in \cite{sos} and \cite{liu}, and to a survey on Tychonoff-type theorems for fuzzy topological spaces \cite{carlson}.

\section{Preliminaries.}
\label{pre}

In this section we shall recall basic notions and results on MV-topological spaces, mainly from \cite{rus2}.

Both crisp and fuzzy subsets of a given set will be identified with their membership functions and usually denoted by lower case latin or greek letters. In particular, for any set $X$, we shall use also $\1$ and $\0$ for denoting, respectively, $X$ and $\varnothing$. In some cases, we shall use capital letters in order to emphasize that the subset we are dealing with is crisp.

We recall that an MV-topological space is basically a special fuzzy topological space in the sense of C. L. Chang \cite{chal}. Moreover, most of the definitions and results of the present section are simple adaptations of the corresponding ones of the aforementioned work to the present context or directly derivable from the same work or from the results presented in the papers \cite{hoh1,hoh2,hoh3,low,rod1,rod2,sto1,sto2}.

\begin{definition}\label{mvtop}
Let $X$ be a set, $A$ the MV-algebra $[0,1]^X$ and $\tau \subseteq A$. We say that $\la X, \tau\ra$ is an \emph{MV-topological space} (or \emph{MV-space}) if $\tau$ is a subuniverse both of the quantale $\la [0,1]^X, \bigvee, \oplus\ra$ and of the semiring $\la [0,1]^X, \wedge, \odot, \1\ra$. More explicitly, $\la X, \tau\ra$ is an MV-topological space if
\begin{enumerate}[(i)]
\item $\0, \1 \in \tau$,
\item for any family $\{o_i\}_{i \in I}$ of elements of $\tau$, $\bigvee_{i \in I} o_i \in \tau$,
\end{enumerate}
and, for all $o_1, o_2 \in \tau$,
\begin{enumerate}[(i)]
\setcounter{enumi}{2}
\item $o_1 \odot o_2 \in \tau$,
\item $o_1 \oplus o_2 \in \tau$,
\item $o_1 \wedge o_2 \in \tau$.
\end{enumerate}
$\tau$ is also called an \emph{MV-topology} on $X$ and the elements of $\tau$ are the \emph{open MV-subsets} of $X$. The set $\xi = \{o^* \mid o \in \tau\}$ is easily seen to be a subquantale of $\la [0,1]^X, \bigwedge, \odot\ra$ (where $\bigwedge$ has to be considered as the join w.r.t. to the dual order $\geq$ on  $[0,1]^X$) and a subsemiring of $\la [0,1]^X, \vee, \oplus, \0\ra$, i.e., it verifies the following properties:
\begin{itemize}
\item[$-$] $\0, \1 \in \xi$,
\item[$-$] for any family $\{c_i\}_{i \in I}$ of elements of $\xi$, $\bigwedge_{i \in I} c_i \in \xi$,
\item[$-$] for all $c_1, c_2 \in \xi$, $c_1 \odot c_2, c_1 \oplus c_2, c_1 \vee c_2 \in \xi$.
\end{itemize}
The elements of $\xi$ are called the \emph{closed MV-subsets} of $X$.
\end{definition}

Let $X$ and $Y$ be sets. Any function $f: X \lto Y$ naturally defines a map
\begin{equation}\label{muf}
\begin{array}{cccc}
f\fcou: & [0,1]^Y & \lto 		& [0,1]^X \\
		 & \alpha  & \lmapsto & \alpha \circ f.
\end{array}
\end{equation}
Obviously $f\fcou(\0) = \0$; moreover, if $\alpha, \beta \in [0,1]^Y$, for all $x \in X$ we have $f\fcou(\alpha \oplus \beta)(x) = (\alpha \oplus \beta)(f(x)) = \alpha(f(x)) \oplus \beta(f(x)) = f\fcou(\alpha)(x) \oplus f\fcou(\beta)(x)$ and, analogously, $f\fcou(\alpha^*) = f\fcou(\alpha)^*$. Then $f\fcou$ is an MV-algebra homomorphism and we shall call it the \emph{MV-preimage} of $f$. The reason of such a name is essentially the fact that $f\fcou$ can be seen as the preimage, via $f$, of the fuzzy subsets of $Y$. From a categorical viewpoint, once denoted by $\Set$, $\BB$ and $\MV$ the categories of sets, Boolean algebras, and MV-algebras respectively (with the obvious morphisms), there exist two contravariant functors $\wp: \Set \lto \BB^{\op}$ and $\fwp: \Set \lto \MV^{\op}$ sending each map $f: X \lto Y$, respectively, to the Boolean algebra homomorphism $f^{-1}: \wp(Y) \lto \wp(X)$ and to the MV-homomorphism $f\fcou: [0,1]^Y \lto [0,1]^X$.

Moreover, for any map $f: X \lto Y$ we define also a map $f^\to: [0,1]^X \lto [0,1]^Y$ by setting, for all $\alpha \in [0,1]^X$ and for all $y \in Y$,
\begin{equation}\label{ovf}
f^\to(\alpha)(y) = \bigvee_{f(x) = y} \alpha(x).
\end{equation}
Clearly, if $y \notin f[X]$, $f^\to(\alpha)(y) = \bigvee \varnothing = \0$ for any $\alpha \in [0,1]^X$.

\begin{definition}\cite{chal}\label{cont}
Let $\la X, \tau_X \ra$ and $\la Y, \tau_Y \ra$ be two MV-topological spaces. A map $f: X \lto Y$ is said to be
\begin{itemize}
\item \emph{continuous} if $f\fcou[\tau_Y] \subseteq \tau_X$,
\item \emph{open} if $f^\to(o) \in \tau_Y$ for all $o \in \tau_X$,
\item \emph{closed} if $f^\to(c) \in \xi_Y$ for all $c \in \xi_X$,
\item an \emph{MV-homeomorphism} if it is bijective and both $f$ and $f^{-1}$ are continuous.
\end{itemize}
\end{definition}
We can use the same words of the classical case because, as it is trivial to verify, if a map between two classical topological spaces is continuous, open, or closed in the sense of the definition above, then it has the same property in the classical sense.

\begin{definition}\cite{war}\label{base}
As in classical topology, we say that, given an MV-topological space $\la X, \tau\ra$, a subset $B$ of $[0,1]^X$ is called a \emph{base} for $\tau$ if $B\subseteq \tau$ and every open set of $(X,\tau)$ is a join of elements of $B$.
\end{definition}

\begin{lemma}\cite{rus2}\label{bascont}
Let $\la X, \tau_X\ra$ and $\la Y, \tau_Y\ra$ be two MV-topological spaces and let $B$ be a base for $\tau_Y$. A map  $f: X \lto Y$ is continuous if and only if $f\fcou[B] \subseteq \tau_X$.
\end{lemma}

A \emph{covering} of $X$ is any subset $\Gamma$ of $[0,1]^X$ such that $\bigvee \Gamma = \1$ \cite{chal}, while an \emph{additive covering} ($\oplus$-covering, for short) is a finite family $\{\alpha_i\}_{i=1}^n$ of elements of $[0,1]^X$, $n < \omega$, such that $\alpha_1 \oplus \cdots \oplus \alpha_n = \1$. It is worthwhile remarking that we used the word ``family'', instead of ``set'', in order to include the possibility for such a family to have repetitions. In other words, an additive covering is a finite subset $\{\alpha_1, \ldots, \alpha_k\}$ of $[0,1]^X$, along with natural numbers $n_1, \ldots, n_k$, such that $n_1\alpha_1 \oplus \cdots \oplus n_k \alpha_k = \1$.

\begin{definition}\label{compact}
An MV-topological space $\la X, \tau \ra$ is said to be \emph{compact} if any open covering of $X$ contains an additive covering; it is called \emph{strongly compact} if any open covering contains a finite covering.\footnote{What we call strong compactness here is called fuzzy compactness in the theory of lattice-valued fuzzy topologies \cite{chal}. It is worth remarking, however, that such a concept appears a very few times in the literature, since it is too much restrictive. Indeed, for example, even a fuzzy topological space with finite support may not be strongly compact.}
\end{definition}

\begin{definition}\cite{martin}\label{t2ax}
Let $\la X, \tau\ra$ be an MV-topological space. $X$ is called a \emph{Hausdorff} (or \emph{separated}) \emph{space} if, for all $x \neq y \in X$, there exist $o_x, o_y \in \tau$ such that
\begin{enumerate}[(i)]
\item $o_x(x) = o_y(y) = 1$,
\item $o_x \wedge o_y = \0$.
\end{enumerate}
\end{definition}

\section{MV-Topologies amongst Fuzzy Topologies.}

In the present section we shall ``place'' MV-topologies within the realm of fuzzy topologies. It is well-known that there exist two main definitions of fuzzy topological space, one due to C. L. Chang and the other to Lowen. Each of them presents advantages and drawbacks, the main differences being probably the following two.
\begin{itemize}
\item Chang's definition includes classical topological spaces as a special case of fuzzy topological ones, in such a way that the category of topological spaces and continuous functions is a full subcategory of the one of fuzzy topologies and fuzzy continuous functions, while the same does not hold for Lowen's fuzzy topologies.
\item In Lowen's fuzzy topology, constant functions are always continuous, while this is not true for Chang's ones.
\end{itemize}

The first definition of fuzzy topology was the one given by Chang.
\begin{definition}\cite{chal}\label{fts}
  Let $\tau\subseteq[0,1]^X$. The pair $(X,\tau)$ is called a \emph{fuzzy topological space}, or \emph{fts} for short, if the following hold:
  \begin{enumerate}[($\tau_1$)]
    \item  $\0, \1\in \tau$;
    \item  $o_1 \wedge o_2 \in \tau$ whenever $o_1, o_2 \in \tau$; and
    \item  $\bigvee_{i \in I} o_i \in \tau$ for any family $\{o_i\}_{i \in I}$ of elements of $\tau$.
  \end{enumerate}
Every member of $\tau$ is called a $\tau$-\emph{open fuzzy set} (or simply open fuzzy set). The complement of a $\tau$-open fuzzy set is called a $\tau$-\emph{closed fuzzy set} (or simply closed fuzzy set).
\end{definition}

In 1976, Lowen gave a definition of fuzzy topological space which included Chang's conditions but, in addition, required all the constant functions to be open fuzzy sets.
\begin{definition}\cite{low}\label{lfts}
 The pair $(X,\tau)$ is called a \emph{laminated fuzzy topological space} if the following conditions hold:
  \begin{enumerate}[($\tau_1$)]
    \item For all $a\in [0,1]$, the $a$-valued constant function $\emph{\textbf{a}}$ is an element of $\tau$;
    \item $o_1 \wedge o_2 \in \tau$ whenever $o_1, o_2 \in \tau$; and
    \item $\bigvee_{i \in I} o_i \in \tau$ for any family $\{o_i\}_{i \in I}$ of elements of $\tau$.
  \end{enumerate}
\end{definition}

The concepts of continuous, open, and closed function, of homeomorphism, and of base are defined exactly as for MV-topologies. Fuzzy topological spaces, with continuous maps, form a category which we denote by $\Fuz$. The category of laminated fuzzy topological spaces and their continuous maps is denoted by $\LFuz$.

In \cite{low} Lowen defined the functors $\iota: \Fuz \lto \Top$ and $\omega: \Top\lto \LFuz$ as follows.

\begin{enumerate}
  \item  $\iota(X,\tau)=\la X,\iota(\tau)\ra$ where $\iota(\tau)$ is the initial topology on $X$ determined by $\tau$ and the lower limit topology on $[0,1]$, that is, $\iota(\tau)$ is the topology generated by the subbase $$B=\{\mu^{-1}[(r,1]]:\mu\in\tau,r\in[0,1)\}=\{\{x\in X:\mu(x)>r\}\}_{\mu\in\tau,r\in[0,1)}.$$
      It is easy to verify that if a map $f: \la X, \tau_X\ra \lto \la Y,\tau_Y\ra$ is continuous, then the map $f: \la X, \iota(\tau_X)\ra\lto \la Y,\iota(\tau_Y)\ra$ is continuous. Moreover, it is worth recalling that $\iota(\tau)$ is the coarsest topology which makes all the elements of $\tau$ lower semi-continuous w.r.t. the usual topology of $[0,1]$.
  \item  $\omega(X,\tau)=\la X,\omega(\tau)\ra$, with $$\omega(\tau)=\bigcup_{r\in[0,1)}\mathcal{C}(X,I_r)=\bigcup_{r\in[0,1)}\{f:X\lto I_r\mid f \text{ is continuous}\},$$ where $I_r=(r,1]$. Note that $\omega(\tau)$ is the set of all lower semicontinuous functions from $\la X,\tau\ra$ to the interval $[0,1]$ equipped with the usual topology. It can be verified that the continuity of a map $f: \la X, \tau_X\ra\lto \la Y,\tau_Y\ra$ implies the continuity of the map $f: \la X, \omega(\tau_X)\ra\lto \la Y,\omega(\tau_Y)\ra$. Thus, $\omega$ is a functor.
\end{enumerate}
A fuzzy space $\la X, \delta\ra$ whose fuzzy topology is of the form $\delta = \omega(\tau)$ for an ordinary topology $\tau$ on $X$ is called \emph{topologically generated} (or \emph{induced}) \cite{low}.

It is possible to define two further functors, $e: \Top\lto \Fuz$ and $j: \Fuz\lto \Top$ \cite{martin}, which will appear in the next sections, in the following way.
\begin{itemize}
\item [3.] $e(X,\tau)=\la X,e(\tau)\ra$, where $$e(\tau)=\{\chi_U:U\in\tau\}.$$ Since the continuity of a map $f: \la X, \tau_X\ra\lto \la Y,\tau_Y\ra$ guarantees the continuity of the map $f: \la X, e(\tau_X)\ra\lto \la Y,e(\tau_Y)\ra$, $e$ is a functor.
\item [4.] $j(X,\tau)=\la X,j(\tau)\ra$, where $$j(\tau)=\tau\cap2^X.$$ Again, the continuity of a map $f: \la X, \tau_X\ra\lto \la Y,\tau_Y\ra$ guarantees the continuity of $f: \la X, j(\tau_X)\ra\lto \la Y,j(\tau_Y)\ra$, thus $j$ is a functor.
\end{itemize}
Note that $j(\tau)$ is the greatest topology contained in $\tau$, and $j$ is exactly the ``skeleton topology functor'' denoted by $\sk$ in \cite{rus2}.

\begin{definition}\cite{martin}
\label{weakly induced}
A fuzzy space $\la X,\tau\ra$ is said to be \emph{weakly induced} if for each $t\in [0,1)$ and for each $\alpha\in \tau$, the characteristic function of $\{x\in X:\alpha(x)>t\}$ belongs to $\tau$.
\end{definition}
Note that a fuzzy space $\la X,\tau\ra$ is weakly induced if and only if $\iota(\tau)=j(\tau)$, and it is topologically generated if and only if is both laminated and weakly induced.

In the following, we recall some useful properties of these functors (see \cite{low} and \cite{cerr}).

\begin{proposition}\label{properties of functors iota, omega, etc}
  \begin{enumerate}[(i)]
    \item For all $\la X,\tau\ra$ in $\Top$, $\iota(\omega(\tau))=\tau.$
    \item $\iota$ is a surjection, $\omega$ is an injection and for each $\tau_1\subseteq \tau_2$, $\iota(\tau_1)\subseteq\iota(\tau_2)$ and $\omega(\tau_1)\subseteq\omega(\tau_2)$ (where $\tau_1,\tau_2$ are topologies or fuzzy topologies, as appropriate).
    \item $\omega(\iota(\tau))$ is the smallest topologically generated fuzzy topology which contains $\tau$.
    \item $\delta$ is topologically generated iff $\omega(\iota(\delta))=\delta$.
    \item $\iota\omega=\iota e=\id_{Top}$.
    \item $\la X,\tau\ra\in e(\Top)$ iff $e(\iota(\tau))=\tau$ iff $\id_X\in \hom_{\Fuz}(e(\iota(X)),X)$.
    \item $\la X,\tau\ra\in \omega(\Top)$ iff $\omega(\iota(\tau))=\tau$ iff $\id_X\in \hom_{\Fuz}(X,\omega(\iota(X)))$.
    \item For all $\la X,\tau\ra$ in $\Fuz$, $\id_X\in \hom_{\Fuz}(\omega(\iota(X)),X)$.
    \item For all $X\in\Fuz$ and for all $Y\in\Top$, $\hom_{\Top}(\iota(X),Y)=\hom_{\Fuz}(X,e(Y))$.
  \end{enumerate}
\end{proposition}

\begin{theorem}\cite{low}
A fuzzy topological space $\la X,\tau\ra$ is topologically generated if and only if for each continuous function $f\in \mathcal{C}(I_r,I_r)$ and for each $\alpha \in \tau$, we have that $f\circ\alpha\in\tau$.
\end{theorem}

\begin{proposition}\label{MV-space is weak induced}
  Each MV-topological space is a weakly induced fuzzy topological space.
\end{proposition}
\begin{proof}
 Let $\la X,\tau\ra$ be an MV-topological space. In the following, we will identify a subset of $A$ with its characteristic map, so we will make no difference between $\chi_A$ and $A$ when $A\subseteq X$. We have to show that for each $t\in [0,1)$ and each $\alpha\in \tau$, the set $\{x\in X:\alpha(x)>t\}$ is an element of $\tau$. We will proceed in three steps.

\begin{claim}
For $\alpha\in\tau$, we have that $\supp\alpha\in\tau.$
\end{claim}
Indeed, as $\supp\alpha=\{x\in X:\alpha(x)>0\}$, then for each $x\in \supp\alpha$, there is some natural number $n$ such that
 $$n\alpha(x)=\underbrace{\alpha(x)\oplus \cdots \oplus \alpha(x)}_{n \text{ times}}=1$$
 then
 $$\chi_{\supp\alpha}=\bigvee_{n=1}^\infty n\alpha\in\tau.$$

\begin{claim}
  For each $\alpha\in \tau$ and every irreducible fraction $t=\frac{k}{2^n}\in (0,1)$,
  $$\alpha_t=\{x\in X:\alpha(x)>t\}\in\tau.$$
\end{claim}
By induction,
\begin{itemize}
  \item For $n=1$ and $k=1$. For each $x\in X$
  $$\alpha(x)>\frac{1}{2} \text{ iff } \alpha(x)\odot\alpha(x)>0.$$
  Hence, $\alpha_{1/2}=\supp(\alpha\odot\alpha)\in\tau.$
  \item Inductive step. Let see that if $\alpha_t\in\tau$ for all $t$ of the form $\frac{k}{2^n}$, then $\alpha_t\in\tau$ for all $t$ of the form $\frac{k}{2^{n+1}}.$

      If $t<1/2$, then $t\oplus t=\frac{k}{2^n},$ hence $\alpha_t=(\alpha\oplus\alpha)_{t\oplus t}\in\tau.$

      If $t\geq1/2$, then $k\geq2^n$ and $t\odot t=\frac{k-2^n}{2^n}$, hence $\alpha_t=(\alpha\odot\alpha)_{t\odot t}\in\tau.$
  \end{itemize}

\begin{claim}\label{claim3}
  For all $\mu \in \tau$ and $t\in [0,1)$,
  $$\mu_t=\{x\in X:\mu(x)>t\}=\bigcup\{\mu_s:s=\frac{k}{2^n},s\geq t\}\in\tau.$$
  \end{claim}
\end{proof}

We will say that an MV-space $\la X,\tau\ra$ is \emph{laminated} if all constant functions on $X$ are elements of $\tau$. It is clear that such spaces form a full subcategory of $\TMV$, which will be denoted by $\LTMV$.

%\subsection{The functors $\omega$, $\iota$, $e$ and $j$ and their relation with $\TMV$}
We will now see some properties about the functors defined in this section, when they are restricted to $\TMV$ both in the domain and in the codomain.

\begin{proposition}\label{functors omega y iota en MVtop prope}
%propiedad 2.4 de cerruti
The functors $\omega$, $\iota$, $e$ and $j$ have the following properties with respect to the category $\TMV$.
  \begin{enumerate}[(i)]
    \item If $\la X,\tau\ra$ is a topological space then $\omega(\tau)$ is an MV-topology, so the codomain of the functor $\omega$ is actually the $\LTMV$ category.
    \item For all $\la X,\tau\ra$ in $\TMV$ and for all $\la Y,\delta\ra$ in $\Top$, $$\hom_{\Top}(Y,\iota(X))=\hom_{\TMV}(\omega(Y),X).$$ This implies that $\omega$ is a left adjoint of $\iota\restr_{\TMV}: \TMV\lto \Top$.
    \item The functor $e$ can be seen as $e:\Top\lto \TMV$.
    \item For all $\la X,\tau\ra$ in $\TMV$ and for all $\la Y,\delta\ra$ in $\Top$, $$\hom_{\Top}(\iota(X),Y)=\hom_{\TMV}(X,e(Y)).$$ This says that $e$ is a right adjoint of $\iota\restr_{\TMV}: \TMV\lto \Top$.
    \item For all $\la X,\tau_X\ra$ in $\TMV$ and for all $\la Y,\tau_Y\ra$ in $\LTMV$,
    $$\hom_{\TMV}(X,Y)\neq\varnothing \Leftrightarrow \la X,\tau_X\ra\in \LTMV.$$
    \item $\LTMV \bigcap \TMV =\omega(\Top).$
  \item For all $\la X,\tau\ra$ in $\TMV$, $\la X, \iota(\tau) \ra = \la X, j(\tau) \ra$.
  	\end{enumerate}
\end{proposition}
\begin{proof}
  \begin{enumerate}[(i)]
    \item We recall that $\omega(\tau)$ is the following fuzzy topology
    $$\omega(\tau)=\bigcup_{r\in[0,1)}\mathcal{C}(X,I_r)=\bigcup_{r\in[0,1)}\{f:X\lto I_r: f \text{ is continuous}\}$$ where $I_r=(r,1]$.
    Let us see that $\omega(\tau)$ is closed for $\oplus$ and $\odot$. If $f:X\lto I_r$ and $g:X\lto I_s$ are elements of $\omega(\tau)$, then
    $f\oplus g:X\lto I_{\min(r,s)}$ given by
    $$(f\oplus g)(x)=f(x)\oplus g(x)=\min(f(x)+g(x),1)$$ is continuous. Analogously,
    $f\odot g:X\lto [0,1]$ given by
    $$(f\odot g)(x)=f(x)\odot g(x)=\max(f(x)+g(x)-1,0)$$ is continuous.
    \item The sentence holds because a function $f:\la Y,\delta\ra\lto \la X,\iota(\tau)\ra$ is continuous in $\Top$ if and only if $f:\la Y,\omega(\delta)\ra\lto\la X,\tau\ra$ is continuous in $\TMV$. Let us see, we have that $f$ is continuous in $\Top$ if for all $t\in[0,1)$ and $\mu\in\tau$, $U=\{x\in X:\mu(x)>t\}\in\iota(\tau)$ implies $$f^{-1}(U)=\{y\in Y:f(y)\in U\}=\{y\in Y:\mu(f(y))>t\}\in\delta.$$ And this is equivalent to say that $\mu\circ f:Y\lto I_t$ is continuous, i. e., $\mu\circ f\in\omega(\delta)$, that is, $f$ is continuous in $\TMV$.
    \item It is clear because $e(\tau)=\{\chi_U:U\in\tau\}$ is an MV-topology whenever $\tau$ is a topology.
    \item It is enough to observe that if $\alpha\in e(\delta)$ then $\alpha=\chi_U$ for some $U\in\delta$. So, $f$ is continuous in $\TMV$ if for all $U\in\delta$, $\chi_U\circ f\in\tau$, that is, $\chi_{f^{-1}(U)}\in\tau$, and it is equivalent to say that $f^{-1}(U)\in\iota(\tau)$, and so $f$ is continuous in $\Top$.
  \item Trivial.
		\item Follows immediately from Proposition \ref{MV-space is weak induced} and \cite[Theorem 2.4]{martin}.
		\item Follows from Proposition \ref{MV-space is weak induced} and the remark right after Definition \ref{weakly induced}.
  \end{enumerate}
\end{proof}

\section{Subbases. Product topologies.}
\label{subprod}

In this section we set the basics for the development of the main topic. We shall define subbases and large subbases for MV-topological spaces, and we will describe the product of MV-topological spaces in complete analogy to the classical case.

\begin{definition}\label{subbase}
Given an MV-topological space $\langle X, \tau\rangle$, a subset $S$ of $\tau$ is called a \emph{subbase} for $\tau$ if each open set of $X$ can be obtained as a join of finite combinations of products, infima, and sums of elements of $S$. More precisely, $S$ is a subbase for $\tau$ if, for all $\alpha \in \tau$, there exists a family $\{t_i\}_{i \in I}$ of terms (or polynomials) in the language $\{\oplus, \odot, \wedge\}$, such that
\begin{equation}\label{subcomb}
\alpha = \bigvee_{i \in I} t_i(\beta_{i1}, \ldots, \beta_{in_i})
\end{equation}
where, for all $i \in I$, $n_i < \omega$, and $\{\beta_{ij}\}_{j=1}^{n_i} \subseteq S$.
\end{definition}

\begin{remark}\label{base-subbase}
If $S$ is a subbase for an MV-topology, the set $B_S$ defined by the following conditions is obviously a base for the same space:
\begin{enumerate}
  \item [(B1)] $S \subseteq B_S$;
  \item [(B2)] if $\alpha, \beta \in B_S$ then $\alpha \star \beta \in B_S$ for $\star \in \{\oplus, \odot, \wedge\}$.
\end{enumerate}
\end{remark}

A subbase $S$ of an MV-topology $\tau$ shall be called \emph{large} if, for all $\alpha \in S$, $n\alpha \in S$ for all $n < \omega$.

\begin{exm}
Let us consider the topology $[0,1]$ on a singleton $\{x\}$. For any $n > 1$, all the sets of type $[0,1/n]$, $[0,1/n] \cap \Q$, and $[0,1/n] \setminus \Q$ are easily seen to be (non-large) subbases for the given topology. Also $[0,1] \setminus \Q$ is a non-large subbase -- a base, in fact.
\end{exm}

\begin{exm}
Let $X$ be a non-empty set and $d: X \lto [0,+\infty[$ be a distance function on $X$. We recall from \cite{low} and \cite{rus2} that, for any fuzzy point $\alpha$ of $X$, with support $x$, and any positive real number $r$, the \emph{open ball} of center $\alpha$ and radius $r$ is the fuzzy set $\beta_r(\alpha)$ defined by the membership function $\beta_r(\alpha)(y) = \left\{\begin{array}{ll} \alpha(x) & \textrm{if } d(x,y) < r \\ 0 & \textrm{if } d(x,y) \geq r \end{array}\right.$, and that the fuzzy subsets of $X$ that are join of a family of open balls form an MV-topology on $X$ that is said to be \emph{induced} by $d$.

It is easy to see that, in such a topology, the set of open balls whose center is a fuzzy point whose non-zero membership value is greater than or equal to some fixed $a < 1$ is a large subbase for the topology induced by $d$. On the contrary, the set of open balls whose center is a fuzzy point whose non-zero membership value is lower than or equal to some fixed $a > 0$ is a non-large subbase.
\end{exm}

\begin{definition}\label{prod}
Let $\{\langle X_i, \tau_i \rangle\}_{i\in I}$ be a family of MV-topological spaces. According to the general definition of Category Theory, we say that an MV-topological space $\la X, \tau\ra$, with a family $(p_i: X \to X_i)_{i \in I}$ of continuous functions, is the \emph{product} of the spaces $\{\langle X_i, \tau_i \rangle\}_{i \in I}$ if, for any MV-topological space $\la Y, \sigma\ra$ and any family of continuous functions $(f_i: Y \to X_i)_{i \in I}$, there exists a unique continuous function $f: Y \to X$ such that $p_i \circ f = f_i$ for all $i \in I$.
\end{definition}

Let $\{\langle X_i, \tau_i \rangle\}_{i\in I}$ be a family of MV-topological spaces. We define the \emph{product MV-topology} $\tau_X$ on the Cartesian product $X=\prod \limits_{i\in I}X_i$ by means of the subbase
\begin{equation}\label{S}
S = \{\pi_i\fcou(\alpha) \mid \alpha \in \tau_i, i \in I\},
\end{equation}
where $\pi_i: X \to X_i$ is the canonical projection. The name ``product MV-topology'' is fully justified by the following result.

\begin{theorem}
The MV-topological space $\langle X, \tau_X\rangle$, with the canonical projections $\pi_i$, is the product of $\{\langle X_i, \tau_i\rangle\}_{i\in I}$.
\end{theorem}
\begin{proof}
First, it is immediate to see that all projections $\pi_i$ are continuous.

Now let $Y$ be an MV-topological space and $(f_i:Y \to X_i)_{i\in I}$ a family of continuous functions. We set $f: y  \in Y \mapsto (f_i (y))_{i\in I} \in X$. Let us show that $f$ is continuous. Let $B$ be the base obtained from $S$ as in Remark \ref{base-subbase} and consider an open set $\beta \in \tau_X$. If $\beta \in S$, namely, $\beta = \alpha \circ \pi_i$ for some $\alpha \in \tau_i$ then, for all $y\in Y$,
$$(\beta \circ f)(y) = ((\alpha \circ \pi_i)\circ f)(y) = (\alpha \circ \pi_i)(f_i(y))_{i\in I} = \alpha(f_i(y)) = (\alpha \circ f_i)(y)$$
and therefore $f\fcou(\beta) = \beta \circ f = \alpha \circ f_i \in \tau_Y$ because each $f_i$ is continuous. Now let us assume that $\beta = \alpha \star \gamma$, with $\star \in \{\oplus, \odot, \wedge\}$ and $\alpha, \gamma \in B$ being such that $\alpha \circ f, \gamma \circ f \in \tau_Y$. Then we have that $\beta \circ f = (\alpha \star \gamma) \circ f = (\alpha \circ f) \star (\gamma \circ f) \in \tau_Y$. Then $f$ is continuous.

Now, in order to prove that $f$ is the universal extension of $(f_i)_{i\in I}$, let $g:Y \to X$ be a continuous function such that $\pi_i \circ g=f_i$ for each $i\in I$. For all $y \in Y$, $g(y) = (\pi_i(g(y)))_{i \in I} = (f_i(y))_{i \in I}$, and therefore $g=f$.
\end{proof}

\section{Tychonoff-type theorem for MV-topologies}
\label{tycho}

In the present section we shall prove the MV-topological correspondents of Alexander Subbase Lemma (Lemma \ref{alex}) and Tychonoff Theorem (Theorem \ref{tych}). As in the classical case, the latter turns out to be an immediate consequence of the former. Moreover, we observe that Theorem \ref{tych} can be obtained also as an immediate consequence of Lowen's Tychonoff theorem for fuzzy topologies and Theorem \ref{ultra}. Nonetheless, besides the fact that the two proofs have been obtained separately and in different moments (hereby presented in chronological order), we believe that the analogous of Alexander Subbase Lemma for MV-topologies is interesting and potentially useful also for future works, therefore we thought it would be the best option to present both the approaches.

\begin{lemma}\label{lem1}
Let $\{\langle X_i, \tau_i\rangle\}_{i\in I}$ be a family of compact MV-topological spaces and let $\langle X, \tau_X\rangle$ be their product. Then any open cover $\Gamma$ of $X$ consisting solely of elements of the form $\alpha \circ \pi_i$, $\alpha \in \tau_i$, contains an additive cover.
\end{lemma}
\begin{proof}
Let $\Gamma$ be such a cover of $X$, and define
$$\Gamma_i=\{\alpha \in \tau_i: \alpha\circ \pi_i \in \Gamma\}.$$

We claim that
\begin{equation}\label{cl1}
\exists j\in I \ \forall x\in X_j \ \exists\alpha_x \in \Gamma_j \ (\alpha_x(x) > 0).
\end{equation}
Indeed, assuming by contradiction that (\ref{cl1}) does not hold, namely, that for each index $i\in I$ there exists $a_i\in X_i$ such that $\alpha(a_i)=0$ for all $\alpha \in \Gamma_i$, then obviously $\left(\bigvee\Gamma_i\right)(a_i)=0$ for all $i\in I$. Therefore, setting $a=(a_i)_{i\in I} \in X$, we get
$$\begin{array}{l}\left(\bigvee\Gamma\right)(a) = \\ = \left(\bigvee_{i\in I}\left(\bigvee_{\alpha \in \Gamma_i} (\alpha \circ \pi_i)\right)\right)(a) = \\ = \left(\bigvee_{i\in I}(\left(\bigvee \Gamma_i\right) \circ \pi_i)\right)(a)  = \bigvee_{i\in I}\left(\bigvee \Gamma_i (a_i)\right) = \\ = 0,\end{array}$$
which implies that $\Gamma$ does not cover $X$, in contradiction with the hypothesis. Hence the statement (\ref{cl1}) holds.

Now, from (\ref{cl1}) it follows that, for all $x \in X_j$, there exists $n_x < \omega$ such that $n_x\alpha_x(x)=1$. Then the family $(n_x\alpha_x)_{x\in X_j}$ is an open cover of $X_j$ and, by the compactness of $X_j$, there exist $x_1,\ldots,x_m \in X_j$ such that
$$\bigoplus_{k=1}^m n_{x_k}\alpha_{x_k} = \1.$$
It follows that
$$\bigoplus_{k=1}^m (n_{x_k}(\alpha_{x_k}\circ\pi_j)) = \bigoplus_{k=1}^m ((n_{x_k}\alpha_{x_k})\circ\pi_j) = \bigoplus_{k=1}^m n_{x_k}\alpha_{x_k} = \1,$$
whence we obtain an additive subcover of $\Gamma$ by simply taking $n_{x_k}$ copies of each $\alpha_{x_k}\circ \pi_j$, $k = 1, \ldots, m$.
\end{proof}

Before proving the MV-topological analogous of Alexander Subbase Lemma, we recall that the following inequality holds in any MV-algebra $A$, for all $a, b ,c \in A$ \cite[Theorem 3.1]{cha1}:
\begin{equation}\label{prop}
a\odot(b\oplus c)\leq b\oplus (a\odot c).
\end{equation}

\begin{lemma}\label{alex}
  Let $\langle X, \tau \rangle$ be an MV-topological space and $S$ a large subbase for $\tau$. If every collection of sets from $S$ that cover $X$ has an additive subcover, then $X$ is compact.
\end{lemma}
\begin{proof}
By contradiction, suppose that every cover of $X$ of elements of $S$ has an additive subcover, and $X$ is not compact. Then the collection
$$\mathfrak{F}=\{\Gamma\subseteq \tau \mid \bigvee \Gamma=\1 \text{ and $\Gamma$ does not contain additive covers}\}$$
is nonempty and partially ordered by set inclusion. We use Zorn's Lemma to prove that $\mathfrak{F}$ has a maximal element. Take any chain $\{E_i\}_{i \in I}$ in $\mathfrak{F}$; let us see that $E=\bigcup E_i$ is an upper bound of such a chain in $\mathfrak{F}$. It is clear that $E\subseteq \tau$ and $\bigvee E=\1$. To see that $E$ contains no additive subcover, look at any finite subcollection $\{f_1,\ldots,f_n\}$ in $E$. Then, for each $k$, there exists $i_k$ such that $f_k \in E_{i_k}$. Since we have a total ordering, there is some $E_{i_0}$ which contains all of the $f_k$'s. Thus such a finite collection cannot be an additive cover. Now, applying Zorn's Lemma, we can assert the existence of a maximal element $M$ in $\mathfrak{F}$.

First of all, let's see some properties of $M$.
\setcounter{claim}{0}
\begin{claim}\label{i)}
$\alpha\notin M$ iff $M \cup \{\alpha\}$ has an additive subcover.
\end{claim}
In other words $\alpha\notin M$ iff there exist $\beta_1,\ldots,\beta_n \in M$ such that $\alpha\oplus\beta_1\oplus\cdots \oplus \beta_n=\1$, and that is obvious.

\begin{claim}\label{ii)}
$\alpha_1, \ldots, \alpha_n \notin M$ implies $\alpha_1 \star \cdots \star \alpha_n \notin M$, for $\star \in \{\wedge, \oplus, \odot\}$.
\end{claim}
\emph{Proof of Claim \ref{ii)}.} First of all note that, for each $i \in \{1, \ldots, n\}$, there exists a finite family $\{\beta_{ij}\}_{j=1}^{m_i}$ of elements of $M$ such that
$$\alpha_i \oplus \bigoplus_{j=1}^{m_i}\beta_{ij}=\1,  \text{ and }  \alpha_i \oplus \bigoplus_{i=1}^n \bigoplus_{j=1}^{m_i}\beta_{ij} = \1.$$
Hence, if we set $\beta \bydef \bigoplus_{i=1}^n \bigoplus_{j=1}^{m_i}\beta_{ij}$, we have $\alpha_i\oplus \beta=\1$ for each $i\in\{1, \ldots, n\}$.

For $\star = \wedge$, for each $x\in X$, we have that $(\alpha_1 \wedge \cdots \wedge \alpha_n)(x)=\alpha_{j_x}(x)$ for some $j_x\in\{1,\ldots,n\}$. So, for each $x\in X$,
$$(\alpha_1 \wedge \cdots \wedge \alpha_n)(x)\oplus \beta(x)=\alpha_{j_x}(x)\oplus \beta(x)=1$$
namely, $\alpha_1 \wedge \cdots \wedge \alpha_n\oplus \beta =\1$, and then $\alpha_1 \wedge \cdots \wedge \alpha_n\notin M$.

Concerning $\odot$, using (\ref{prop}), we have that
$$\bigodot_{i=1} ^n\alpha_i \oplus \beta\geq\bigodot_{i=1}^{n-1}\alpha_i \odot (\alpha_n \oplus \beta)=\bigodot_{i=1} ^{n-1}\alpha_i\odot \1=\bigodot_{i=1} ^{n-1}\alpha_i$$
    then
$$ \bigodot_{i=1} ^n\alpha_i \oplus \beta\oplus \beta\geq\bigodot_{i=1} ^{n-1}\alpha_i \oplus \beta\geq \bigodot_{i=1} ^{n-2}\alpha_i$$
    whereby
    $$\bigodot_{i=1} ^n\alpha_i \oplus \underbrace{\beta\oplus \cdots \oplus \beta}_{n-1 \text{ times}}\geq \alpha_1$$
    and therefore
    $$\bigodot_{i=1} ^n\alpha_i \oplus \underbrace{\beta\oplus \cdots \oplus \beta}_{n \text{ times}}\geq \alpha_1\oplus \beta=\1.$$
It follows that
$(\alpha_1 \odot \cdots \odot \alpha_n)\oplus \underbrace{\beta\oplus \cdots \oplus \beta}_{n \text{ times}}=\1$ where $\beta\in M$ and then $\nobreak{\alpha_1 \odot \cdots \odot \alpha_n\notin M}$.

Last, for $\star = \oplus$, if $\alpha_1,\ldots,\alpha_n\notin M$ then, in particular, $\alpha_1\oplus\beta=\1$. It follows that $\alpha_1 \oplus \cdots \oplus \alpha_n \oplus\beta=\1$ and, therefore, $\alpha_1 \oplus \cdots \oplus \alpha_n \notin M$.

\begin{claim}\label{iii)}
If $\alpha \notin M$ and $\alpha \leq \beta$ then $\beta \notin M$.
\end{claim}
\emph{Proof of Claim \ref{iii)}.} Indeed, if $\alpha \notin M$ there exist $\beta_1,\ldots,\beta_n \in M$ such that $\alpha\oplus\beta_1\oplus\cdots \oplus \beta_n=\1$ and then $\1=\alpha\oplus\beta_1\oplus\cdots \oplus \beta_n\leq\beta\oplus\beta_1\oplus\cdots \oplus \beta_n$, so $\beta\notin M$.
    %$$M$ is an ideal of $[0,1]^X$

\begin{claim}\label{iv)}
$M$ is an ideal of the MV-algebra $[0,1]^X$.
\end{claim}
\emph{Proof of Claim \ref{iv)}.} $M$ is non-empty and, if $\alpha \in M$ and $\beta \leq \alpha$, then $\beta \in M$ by Claim \ref{iii)}. Moreover, if $\alpha, \beta\in M$ then $\alpha\oplus\beta\in M$ because, otherwise, if $\alpha\oplus\beta\notin M$ there exist $\beta_1,\ldots,\beta_n \in M$ such that $\alpha\oplus\beta\oplus\beta_1\oplus\cdots \oplus \beta_n=\1$. But this is impossible because $M$ does not contain additive subcovers.

Observe that, as a consequence of Claims \ref{ii)} and \ref{iii)}, the set $F=\{\beta\in \tau:\beta\notin M\}$ is a filter of the MV-algebra $[0,1]^X$.

Let us now consider the set $T=M\cap S$, and let us prove that $T$ is a cover of $X$. Since $M$ is a covering of $X$, for each $a\in X$ there exists $\alpha_a \in M$ such that $\alpha_a(a)>0$. On the other hand, since $S$ is a subbase, there exists a family $\{t_i\}_{i \in I}$ of terms (or polynomials) in the language $\{\oplus, \odot, \wedge\}$, such that
\begin{equation}\label{alpha_a}
\alpha_a = \bigvee_{i \in I} t_i(\beta_{i1}, \ldots, \beta_{in_i})
\end{equation}
where, for all $i \in I$, $n_i < \omega$, and $\{\beta_{ij}\}_{j=1}^{n_i} \subseteq S$.

\begin{claim}\label{claim 1}
Let $t$ be a term in the language $\{\oplus, \odot, \wedge\}$, $\{\beta_1,\ldots,\beta_n\}\subseteq S$, $\ov \beta = (\beta_1,\ldots,\beta_n)$, and $t(\ov\beta)\in M$. If $t(\ov \beta)(a)>0$ for some $a \in X$, then there exists $j\in\{1,\ldots,n\}$ such that $\beta_j\in M$ and $\beta_j(a)>0$.
\end{claim}
\emph{Proof of Claim \ref{claim 1}.} Let us proceed by induction on the length of the term $t$. If $t$ has length 1, then $t(\ov \beta)=\beta$ with $\beta\in S$, and the condition clearly holds.

Suppose for inductive hypothesis that the assertion holds for all term of length $< m$, and let $t(\ov \beta)\in M$ be a term of length $m$ such that $t(\ov \beta)(a)>0$ for some $a \in X$. Since $t$ has length $m$ then $t=r \star s$, where $r$ and $s$ are terms of length $< m$ and $\star \in \{\wedge, \odot,\oplus\}$. Then we have to distinguish three cases.

If $t=r\wedge s$ then $t(\ov \beta)(a)=r(\ov \beta)(a)\wedge s(\ov \beta)(a)$, so $r(\ov \beta)(a)>0$ and $s(\ov \beta)(a)>0$ because $t(\ov \beta)(a)>0$. Furthermore, since $(r\wedge s)(\ov \beta) \in M$, by Claim \ref{ii)}, $r(\ov \beta)\in M$ or $s(\ov \beta)\in M$. Without loss of generality we can assume that $r(\ov \beta)\in M$, then for inductive hypothesis we have that there exists $j\in\{1,\ldots,n\}$ such that $\beta_j\in M$ and $\beta_j(a)>0$. So the assertion holds for $t= r\wedge s$.

If $t=r\odot s$ then $t(\ov \beta)(a)=r(\ov \beta)(a)\odot s(\ov \beta)(a)$, so $r(\ov \beta)(a)>0$ and $s(\ov \beta)(a)>0$ because $t(\ov \beta)(a)>0$. As in the previous case, $r(\ov \beta)\in M$ or $s(\ov \beta)\in M$ for Claim \ref{ii)}. Without loss of generality we can assume that $r(\ov \beta)\in M$, then for inductive hypothesis we have that there exists $j\in\{1,\ldots,n\}$ such that $\beta_j\in M$ and $\beta_j(a)>0$. So Claim \ref{claim 1} holds if $t= r \odot s$.

Last, if $t=r\oplus s$ then $t(\ov \beta)(a)=r(\ov \beta)(a)\oplus s(\ov \beta)(a)$, so $r(\ov \beta)(a)>0$ or $s(\ov \beta)(a)>0$ because $t(\ov \beta)(a)>0$. Furthermore $r(\ov \beta)\in M$ and $s(\ov \beta)\in M$, because $t(\ov \beta)\in M$, $r(\ov \beta),s(\ov \beta)\leq t(\ov \beta)$, and $M$ is an ideal. Without loss of generality we can assume that $r(\ov \beta)(a)>0$, then for inductive hypothesis we have that there exists $j\in\{1,\ldots,n\}$ such that $\beta_j\in M$ and $\beta_j(a)>0$. So the assertion holds also for $t= r\oplus s$, and this completes the proof of Claim \ref{claim 1}.

Now, from the representation of $\alpha_a$ in (\ref{alpha_a}) we have that $t_i(\beta_{i1}, \ldots, \beta_{in_i})\in M$ for all $i\in I$, because $t_i(\beta_{i1}, \ldots, \beta_{in_i})\leq \alpha_a$ for each $i\in I$ and $\alpha_a$ is an element of the ideal $M$.

Moreover, there exists $j\in I$ such that $t_j(\beta_{j1}, \ldots, \beta_{jn_j})(a)>0$ because $\alpha_a(a)>0$. Then, by Claim \ref{claim 1}, we have that there exists $\beta_a=\beta_{jk}\in M$ with $k\in\{1,\ldots,n_j\}$ such that $\beta_a(a)>0$. Therefore we get $n_a\beta_a(a)=1$ for some $n_a<\omega$.

It means that the family $\{n_a\beta_a\}_{a\in X}$ is a covering of $X$ which is contained in $T=M\cap S$. From the hypothesis about $S$ we have that $T$ has an additive subcover, so there exists a finite subset $\{n_{a_1}\beta_{a_1},\ldots,n_{a_t}\beta_{a_t}\}$ of $T$ such that $n_{a_1}\beta_{a_1}\oplus \cdots \oplus n_{a_t}\beta_{a_t}=1$. But this means that $M$ has an additive subcover too, which is a contradiction.

Therefore, our original collection $\mathfrak{F}$ must be empty, whence $X$ is compact.
\end{proof}

\begin{theorem}[Tychonoff-type Theorem for MV-Topologies]\label{tych}
  If $\{\langle X_i, \tau_i\rangle\}_{i\in I}$ is a family of compact MV-topological spaces, then so is their product space $\langle X, \tau_X\rangle$.
\end{theorem}
\begin{proof}
  Let us consider as a subbase for the product MV-topology on $X$ the collection
  $$S=\{\pi_i\fcou(\beta):\beta\in \tau_i, i\in I\}.$$

$S$ is a large subbase; indeed, for each $n<\omega$, $n(\beta\circ\pi_i)=n\beta\circ\pi_i$, and $n\beta\in \tau_i$ whenever $\beta\in\tau_i$. By Lemma \ref{lem1}, any subcollection of $S$ that covers $X$ has an additive subcover. Then the compactness of $X$ follows from Lemma \ref{alex}.
\end{proof}

\begin{remark}\label{tych2}
Theorem \ref{tych} can be obtained also as a corollary of the following two results.
\end{remark}

\begin{theorem}\label{ultra}
Every MV-topological space $\la X,\tau\ra$ is compact if, and only if it is ultra-fuzzy compact in the sense of Lowen \cite{low3}, i.e., the topological space $\la X,\iota(\tau)\ra$ is compact.
\end{theorem}
\begin{proof}
The ``only if'' part is trivial. For what concerns the converse implication, suppose that $\la X,\iota(\tau)\ra$ is a compact topological space and $\{\alpha_i:i\in I\}$ is an open cover of $X$. For each $\beta\in \tau$ and $t\in [0,1)$, let $\beta_t=\{x\in X:\beta(x)>t\}$. Since the family $\{(\alpha_i)_{\frac{1}{2}}:i\in I\}$, is an open cover of the topological space $\la X,\iota(\tau)\ra$, there exists a finite subfamily $\{(\alpha_{i_1})_{\frac{1}{2}}, \ldots,(\alpha_{i_m})_{\frac{1}{2}}\}$ that covers $X$. This means $\{\alpha_{i_1}, \ldots,\alpha_{i_m}\}$ is an additive open cover of $\la X,\tau\ra$.
\end{proof}

\begin{theorem}\cite[Theorem 3.3]{low3}\label{low333}
  Let $\{\la X_i, \tau_i\ra\}_{i\in I}$ be a family of fuzzy topological spaces. The product space $\la\prod \limits_{i\in I}X_i,\tau\ra$ is ultra-fuzzy compact if and only if for all $i\in I$, $\la X,\tau_i\ra$ is ultra-fuzzy compact.
\end{theorem}

We conclude this section by stating the following corollary, which is an immediate consequence of Proposition \ref{MV-space is weak induced}.

\begin{corollary}\label{compcor}
An MV-topological space $\la X, \tau\ra$, is compact in the sense of Definition \ref{compact} iff $\la X, \iota(\tau)\ra$ or, that is the same, $\la X, j(\tau)\ra$ is a compact topological space.
\end{corollary}

\section{Some consequences of Tychonoff theorem.}
\label{cons}

Let us now briefly discuss some consequences of Theorem \ref{tych}. The first two results of the section are of independent interest, and necessary for establishing the subsequent two corollaries.

\begin{lemma}\label{t2}
The product of Hausdorff MV-topologies is Hausdorff.
\end{lemma}
\begin{proof}
The proof proceeds analogously to the classical case with no major differences. Indeed, let $\{\langle X_i, \tau_i\rangle\}_{i\in I}$ be a family of Hausdorff MV-spaces, $\la X,\tau\ra$ its product space, and $(x_i)_{i \in I}, (y_i)_{i \in I}$ two distinct points of $X$. So there exists $j \in I$ such that $x_j \neq y_j$ and, since every $X_i$ is Hausdorff, there exist $o_x, o_y \in \tau_j$ such that $o_x(x_j) = o_y(y_j) = 1$ and $o_x \wedge o_y = \0$. Then it is not hard to see that the open sets $o_x \circ \pi_j$ and $o_y \circ \pi_j$ separate the given points of $X$, namely, $(o_x \circ \pi_j)((x_i)_{i \in I}) = (o_y \circ \pi_j)((y_i)_{i \in I}) = 1$ and $(o_x \circ \pi_j) \wedge (o_y \circ \pi_j) = \0$.
\end{proof}

\begin{lemma}\label{0dim}
The product of zero-dimensional MV-topological spaces is zero-dimen\-sional.
\end{lemma}
\begin{proof}
Since sums, products, and finite infima of clopens of an MV-topological space are clopens, the assertion follows immediately from (\ref{S}) and Remark \ref{base-subbase}.
\end{proof}

\begin{corollary}\label{stone}
The product of Stone MV-spaces is a Stone MV-space.
\end{corollary}

\begin{corollary}\label{lcc}
The category $\MVlcc$, of limit cut complete MV-algebras and MV-algebra homomorphisms, has coproducts.
\end{corollary}
\begin{proof}
It is an immediate consequence of Theorem \ref{tych}, Lemmas \ref{t2} and \ref{0dim}, and the duality between $\MVlcc$ and $\SMV$ \cite[Theorem 4.9]{rus2}.
\end{proof}

It is important to observe that Corollary \ref{lcc} does not guarantee that the coproduct, in $\MV$, of lcc MV-algebras is lcc too. Moreover, as Mundici observed in \cite[Corollary 7.4]{mubook}, the classes of totally ordered, hyperarchimedean, simple, and semisimple MV-algebras are not preserved under coproducts in the category of MV-algebras.

In order to better understand coproducts of lcc MV-algebras we prove the following result.
\begin{proposition}\label{radcoprod}
Let $(A_i)_{i \in I}$ be a family of lcc MV-algebras, and let $A, A'$, and $A''$ be the coproducts of such a family in $\MVlcc$, $\MVs$, and $\MV$, respectively. Then we have $A \cong A' \cong A''/\Rad A''$.
\end{proposition}
\begin{proof}
Let $(\mu_i)_{i \in I}$, $(\nu_i)_{i \in I}$, and $(\eta_i)_{i \in I}$ be, respectively, the embeddings of the given family in $A, A'$, and $A''$. For any semisimple MV-algebra $B$ and morphisms $(f_i: A_i \to B)_{i \in I}$, there exists a morphism $f: A'' \to B$ such that $f\eta_i = f_i$ for all $i \in I$. The semisimplicity of $B$ guarantees that $\ker f \subseteq \Rad A''$ and, therefore, there exists a morphism $g: A''/\Rad A'' \to B$ such that $g\pi = f$, where $\pi$ is the canonical projection of $A''$ over $A''/\Rad A''$. So, for all $i \in I$, $g\pi\eta_i = f\eta_i = f_i$. Therefore $A''/\Rad A''$ is the coproduct in $\MVs$ of $(A_i)_{i \in I}$, with embeddings $(\pi\eta_i)_{i \in I}$, whence $A' \cong A''/\Rad A''$.

Now, by \cite[Corollary 5.8]{rus2}, the lcc completion $(A')^{\lcc}$ of $A'$ is also a coproduct of the family $(A_i)_{i \in I}$ in $\MVlcc$. Therefore, if we denote by $\iota: A' \to (A')^{\lcc}$ the inclusion morphism, by $\mu: A' \to A$ the morphism such that $(\mu\nu_i)_{i \in I} = (\mu_i)_{i \in I}$, and by $\ov\mu$ the unique extension of $\mu$ to $(A')^{\lcc}$ as in \cite[Corollary 5.8]{rus2}, we get that $\ov\mu$ is an isomorphism for the essential uniqueness of coproducts in any given category, and $\mu$ is onto because it is surjective on a generating set of $A$. On the other hand, the families $(\mu_i)_{i \in I}$, $(\nu_i)_{i \in I}$, and $(\iota\nu_i)_{i \in I}$ are right-cancellable, for being epi-sinks. It follows that $\ov\mu\iota = \mu$, i.e., $\mu$ is injective too. Then $\mu$ is an isomorphism, and we get $A'= (A')^{\lcc}$, $\iota = \id_{A'}$, $\mu = \ov\mu$,  and $A \cong A'$. The diagram below will better illustrate the last part of the proof.
$$
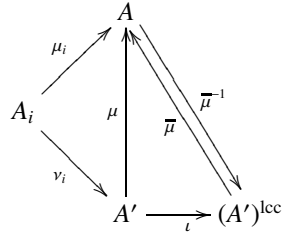
\xymatrix{
& A \ar@<.5ex>[rdd]^{\ov\mu^{-1}}& \\
A_i \ar[rd]_{\nu_i} \ar[ru]^{\mu_i}& &\\
& A' \ar[uu]^\mu \ar[r]_\iota & (A')^{\lcc} \ar@<.5ex>[luu]^{\ov\mu}
}
$$
\end{proof}

In \cite{low2} the author proved Tychonoff theorem for lattice-valued fuzzy topology. Theorem \ref{tych} obviously imply classical Tychonoff theorem because every classical topological space is an MV-topological space too. On the other hand, it is known that the same holds -- although less obviously -- for Lowen's result, as we show in the next proposition which can be easily deduced from the results in \cite{low3}.

We recall that there exists a categorical full embedding $\omega: \Top \to \LFuz$, of the category of topological spaces and continuous functions into the one of fuzzy topologies in the sense of Lowen, with fuzzy continuous functions, which associates, to each topological space, the so-called \emph{topologically generated} fuzzy topological space \cite{low}.

\begin{proposition}\label{lowth}
Lowen's Tychonoff theorem implies Tychonoff theorem.
\end{proposition}
\begin{proof}
We need to prove that, if the product of every family of fuzzy compact topological spaces is fuzzy compact, then the product of every family of compact topological spaces is compact. In order to do that, we recall the following facts.
\begin{itemize}
\item Fuzzy compactness is a \emph{good} fuzzy topological property, namely, a topologically generated fuzzy topological space is compact iff the underlying topological space is compact \cite[Theorem 4.1]{low} and \cite[Theorem 2.1]{low3}.
\item The $\omega$ functor commutes with products \cite[Corollary 3.7]{sos2}.
\end{itemize}
Let $\{\la X_i, \tau_i\ra\}_{i \in I}$ be a family of compact topological spaces. Since compactness is a good property, the topologically generated fuzzy spaces of the family $\{\la X_i, \omega \tau_i\ra\}_{i \in I}$ are compact. On the other hand, the product of such fuzzy spaces is topologically generated by the product of the spaces $X_i$, because $\omega$ commutes with products. By Lowen's theorem, such a product is fuzzy compact. Then the product of the $X_i$ is compact, again, because compactness is a good property.
\end{proof}

Both classical and Lowen's Tychonoff theorems need the Axiom of Choice, which is known to be equivalent to classical Tychonoff theorem in ZF \cite{kel}. Therefore, the following equivalence holds.

\begin{theorem}\label{equiv}
The following statements are equivalent in ZF:
\begin{enumerate}[(a)]
\item the Cartesian product of a non-empty family of non-empty sets is non-empty (AC);
\item the product space of compact topological spaces is compact \cite{ty};
\item the product space of fuzzy compact topological spaces (in the sense of \cite{low}) is fuzzy compact \cite{low2};
\item the product space of compact MV-topological spaces is compact (Theorem \ref{tych}).
\end{enumerate}
\end{theorem}

\section{Compactification}
\label{compactification}

In 1981, Cerruti \cite{cerr} studied some concepts of fuzzy topological spaces from the categorical point of view, and developed a compactification theory. In order to do that, he showed the existence of a left adjoint functor to the embedding $e:\HCAFuz\lto \Fuz$ where $\HCAFuz$ is the category of compact Hausdorff weakly induced spaces. We present an analogous categorical proof on the MV-topological spaces in the present section.

In what follows, $\CTMV$ will denote the full subcategory of $\TMV$ whose objects are compact MV-spaces, and $\HCTMV$ the full subcategory whose objects are compact Hausdorff MV-spaces.

\subsubsection*{The Stone-\v Cech Compactification}

\begin{lemma}\cite{rus2}\label{closcomp}
A closed subspace $(Y, \tau_Y)$ of a compact (respectively: strongly compact) space $(X, \tau)$ is compact (resp.: strongly compact).
\end{lemma}
\begin{proof}
Since $Y$ is a subspace, in particular it is a crisp subset of $X$ and, therefore, all of its open sets are of the form $\alpha\restr_{Y}$ with $\alpha \in \tau$. So let $\{\alpha_i\}_{i \in I} \subseteq \tau$ such that $\bigvee_{i \in I} \alpha_i \geq Y$. Since $Y$ is closed, $Y^*$ is open and $\{\alpha_i\}_{i \in I} \cup \{Y^*\}$ is an open covering of $X$. By compactness of $X$, there exists a finite family $\{\alpha_j\}_{j=1}^n$ of elements of $\{\alpha_i\}_{i \in I}$ such that $\alpha_1 \oplus \cdots \oplus \alpha_n \oplus Y^* = X$. Then, since $Y \wedge Y^* = \tau$, we have (with a slight abuse of notation) $Y = Y \wedge (\alpha_1 \oplus \cdots \oplus \alpha_n) = (Y \wedge \alpha_1) \oplus \cdots \oplus (Y \wedge \alpha_n)$, the latter equality easily following from the properties of Boolean elements of MV-algebras, whence $Y$ is compact.

The case of strong compactness is completely analogous.
\end{proof}

\begin{lemma}\label{equalizer closed if Y is Hausdorff}
  Let $\la X,\tau_X\ra$ and $\la Y,\tau_Y\ra$ be MV-spaces and let $f,g: X\to Y$ be continuous functions. If $Y$ is a Hausdorff MV-space then the set $$Z=\{x\in X:f(x)=g(x)\}$$ is a closed crisp subset.
\end{lemma}
\begin{proof}
  Let $x\in X\setminus Z$, so $f(x)\neq g(x)$. Since $Y$ is Hausdorff, there exist $\alpha,\beta\in\tau_Y$ such that $\alpha(f(x))=\beta(g(x))=1$ and $\alpha\wedge\beta=\0$. Moreover, $f$ and $g$ are continuous, so we have that $\alpha\circ f$ and $\beta\circ g$ are open sets of $(X,\tau_X)$. Set $\gamma_x=(\alpha\circ f)\wedge(\beta\circ g)$. Then $\gamma_x\in \tau_X$, $\gamma_x(x)=1$, and
	$$\gamma_x(z)=((\alpha\circ f)\wedge(\beta\circ g))(z)= \alpha(f(z))\wedge\beta(g(z))=(\alpha\wedge\beta)(f(z))=0,$$
	for each $z\in Z$
	
  It follows that $Z^*= \bigvee_{x\in X\setminus Z}\gamma_x\in\tau_X$, whence $Z$ is closed.
\end{proof}

\begin{proposition}The category $\HCTMV$ satisfies the following properties:
  \begin{enumerate}[(a)]
    \item $\HCTMV$ has all products.
    \item $\HCTMV$ has equalizer.
    \item $\HCTMV$ has a small cogenerator.
  \end{enumerate}
\end{proposition}

\begin{proof}
  \begin{enumerate}[(a)]
     \item Follows from Theorem \ref{tych} and Lemma \ref{t2}.
     \item Let $f,g:X\lto Y$ be morphisms in $\HCTMV$. Seeing these morphisms in $\Set$, we know that $Z=\{x\in X: f(x)=g(x)\}$ is the equalizer of them. Now, since $Y$ is a Hausdorff space, $Z$ is closed in $X$ by Lemma \ref{equalizer closed if Y is Hausdorff}. So $Z$ is a compact MV-space (Lemma \ref{closcomp}) and the canonical injection $m:Z\lto X$ is the equalizer. Note that $Z$ is an element of $\HCTMV$.
     \item Let us consider the interval $I$ in $\Top$ with the usual topology and show that the cogenerator in $\HCTMV$ is $e(I)$. Indeed, let $X$ be an element of $\HCTMV$, $x,y\in X$, $x\neq y$. As $\iota(X)$ is a compact Hausdorff space, there exists a morphism $f:\iota(X)\lto I$ in $\Top$ such that $f(x)\neq f(y)$. By Proposition \ref{functors omega y iota en MVtop prope}, $f:X\lto e(I)$ is a morphism in $\TMV$.
   \end{enumerate}
\end{proof}

Let now $i$ be the inclusion functor $\HCTMV\hookrightarrow \TMV$. From (a) and (b) of the last proposition we have that $\HCTMV$ is small-complete and we obtain the following corollary.
\begin{theorem}
  The functor $i: \HCTMV\lto \TMV$ has a left adjoint.
\end{theorem}
\begin{proof}
$\HCTMV$ is small-complete and has a small cogenerator, therefore the assertion follows from the Special Adjoint Functor Theorem.
\end{proof}

We denote by $\widehat{\beta}:\TMV\lto\HCTMV$ the left adjoint functor of $i$, and the adjunction by $\widehat{\beta}\dashv i$. Note that $\HCTMV$ is a reflective subcategory of $\TMV$, then we have that each object $X$ of $\HCTMV$ is isomorphic to its reflection, that is, $X\simeq \widehat{\beta} (X)$.

We will show now that $\widehat{\beta}$ is the natural extension of the classical Stone-\v Cech compactification to the category $\TMV$. In what follows, $\Top$ and $\HCTop$ will denote the categories of topological spaces and Hausdorff compact topological spaces, respectively, both with the usual morphisms, and $\beta$ shall denote the Stone-\v Cech compactification functor between them.

\begin{theorem}\label{adjunctions compactification}
The functors $\beta$ and $\iota\widehat{\beta}\omega$ are naturally isomorphic.
\end{theorem}
\begin{proof}
  Let us consider the following adjunctions:
  \begin{enumerate}[(i)]
    \item $\omega$ is a left adjoint of $\iota\restr_{\TMV}: \TMV\lto\Top$ (see Proposition \ref{functors omega y iota en MVtop prope} (ii)),
    \item $\widehat{\beta}$ is a left adjoint of $i$, and
    \item $\iota\restr_{\HCTMV}$ is a left adjoint of $e$ (see Proposition \ref{functors omega y iota en MVtop prope} (iv)).
  \end{enumerate}
  From the following compositions, we obtain the adjunction $\iota\widehat{\beta}\omega\dashv \iota i e$:
  $$\Top\stackrel{\omega}{\lto}\TMV\stackrel{\widehat{\beta}}{\lto}\HCTMV\stackrel{\iota}{\lto}\HCTop,$$
	$$\HCTop\stackrel{e}{\lto}\HCTMV\stackrel{i}{\lto}\TMV\stackrel{\iota}{\lto}\Top.$$
	
  On the other hand, $\iota i e=i:\HCTop\lto \Top$, therefore $\iota\widehat{\beta}\omega$ is a left adjoint of the embedding of $\HCTop$ in $\Top$, then it is naturally isomorphic to $\beta$.
\end{proof}

Let us now show that, for an MV-space $X$, the initial topology of $X$ determines the initial topology of the \emph{MV-compactification} of $X$.

\begin{theorem}\label{iota de compacti es compactif de iota}
For each $X$ in $\TMV$, $\iota\widehat{\beta}(X)\cong \beta\iota(X).$
\end{theorem}
\begin{proof}
 First of all, the adjunction $\iota\widehat{\beta}\dashv i e$ can be obtained by composing the adjunctions $\widehat{\beta}\dashv i$ and $\iota\dashv e$ as follows:
 $$\TMV\stackrel{\widehat{\beta}}{\lto}\HCTMV\stackrel{\iota}{\lto}\HCTop,$$
 $$\HCTop\stackrel{e}{\lto}\HCTMV\stackrel{i}{\lto}\TMV.$$
It is enough to show that the restriction of $\beta\iota$ on $\TMV$ is a left adjoint of $ie=e.$

From Proposition \ref{functors omega y iota en MVtop prope} (iv), we have that $\iota\restr_{\TMV}: \TMV\lto \Top$ is a left adjoint of $e$, that is, for all $\la X,\tau\ra$ in $\TMV$ and for all $\la Y,\delta\ra$ in $\Top$, $$\hom_{\Top}(\iota(X),Y)=\hom_{\TMV}(X,e(Y)).$$
Since $\hom_{\Top}(\iota(X),Y)\cong\hom_{\Top}(\beta\iota(X),Y)$, then $$\hom_{\TMV}(X,e(Y))\cong\hom_{\Top}(\beta\iota(X),Y),$$ whence the thesis follows.
\end{proof}

As a consequence of the last result we have that, for each MV-space $X$, the canonical morphism $\eta_X:X\lto\widehat{\beta}(X)$ has the same underlying map of the canonical morphism $\iota(X)\lto\beta\iota(X).$ %Besides, $\C$ is an epireflective subcategory of \textbf{HMV-TOP}.

Let us also observe explicitly that $\widehat{\beta}$ is basically the restriction of the functor $\widetilde{\beta}$ introduced by Cerruti in \cite{cerr}, since it is just the left adjoint to $i$, which is the restriction of the functor $e:\HCAFuz\lto \Fuz$ to MV-topological spaces. So we have the following commutative diagram:
$$\xymatrix{
\Top \ar[r]^\omega \ar[d]_\beta & \TMV \ar[r]^\subseteq \ar[d]^{\widehat{\beta}} & \Fuz \ar[d]^{\widetilde{\beta}} \\
\HCTop \ar[r]_\omega & \HCTMV \ar[r]_\subseteq & \HCAFuz
}$$

Finally, we have the following result:
\begin{theorem}\label{topgencomp}
  \begin{enumerate}[(i)]
    \item $\widehat{\beta} e=e\beta.$
    \item If $X$ is topologically generated then $\widehat{\beta}(X)=\omega\beta\iota(X).$
  \end{enumerate}
\end{theorem}
\begin{proof}
  \begin{enumerate}[(i)]
  \item Trivial.
  \item [(ii)] Let $X$ be a topologically generated space, then $X=\omega\iota(X)$ by Proposition \ref{properties of functors iota, omega, etc} (iv). Given a morphism $\varepsilon_X:\iota(X)\lto \beta\iota(X)$, we have that $$\varepsilon_X\in \hom_{\TMV}(\omega\iota(X),\omega\beta\iota(X))=\hom_{\TMV}(X,\omega\beta\iota(X)).$$
  Since $\omega\beta\iota(X)$ is an object of $\HCTMV$, there exists a unique $f$ which makes the following diagram commute
  \begin{equation*}
\xymatrix{
X \ar@{->}[rr]^{\eta_X} && \widehat{\beta}(X) \\
&&\\
&& \omega\beta\iota(X) \ar@{<--}[uu]_f \ar@{<-}[lluu]^{\varepsilon_X}\\
}
\end{equation*}

In other words, $f\in\hom_{\TMV}(\widehat{\beta}(X), \omega\beta\iota(X))$ and, from Proposition \ref{functors omega y iota en MVtop prope}(v,vi), we have that $\widehat{\beta}(X)\in\omega(\Top)$. Thus $\widehat{\beta}(X)=\omega\iota\widehat{\beta}(X)$, and $\widehat{\beta}(X)=\omega\iota\widehat{\beta}(X)=\omega\beta\iota(X)$ by the inclusion $\omega(\Top)\subseteq \TMV$ and Theorem \ref{iota de compacti es compactif de iota}.
  \end{enumerate}
\end{proof}

From Theorem \ref{topgencomp}(ii) it follows that $\widehat{\beta}\omega=\omega\beta$. Indeed, for each $X\in \TMV$, $\widehat{\beta}\omega(X)=\omega\beta\iota\omega(X)=\omega\beta(X)$.

On the other hand, using (v) and (iv) of Proposition \ref{functors omega y iota en MVtop prope}, we also get that an MV-space $X$ is topologically generated whenever $\widehat{\beta}(X)$ is topologically generated, because $\eta_X$ is an element of $\hom_{\TMV}(X,\widehat{\beta}(X))$.

From the results of this section, it follows that $\widehat{\beta}$ is an extension of the classical Stone-\v Cech Compactification $\beta$, and that the two functors enjoy similar properties.

Here we concentrated on the Stone-\v Cech compactification functor from the category of {\bf all} MV-topological spaces to the one of compact Hausdorff ones and, therefore, the main reference for our construction was \cite{cerr}. Compactifications of fuzzy topological spaces have been studied, with a different approach, also in \cite{martin0, martin} and \cite{martin2}. In those papers, the author presents several results on how to homeomorphically and densely embed fuzzy topological spaces into compact separated ones. Here follows some considerations about our work and Martin's ones.

First of all, it is important to remark that MV-topologies are very well-behaved w.r.t. the fuzzyfication of various classical topological concepts. For example, for any MV-topological space, the properties of being Hausdorff (Definition \ref{t2ax}), ultra Hausdorff \cite[Definition 3.1]{martin0}, and $T_2$ \cite[Definition 6.4]{puliu} are all equivalent, and also the various definitions of compactness, except strong compactness as defined in \cite{rus2}, collapse to a single property. The former equivalence can be easily seen thanks to Remark 3.8 in \cite{rus2}, Proposition \ref{MV-space is weak induced}, and the fact that the support of any fuzzy open set is open itself in any MV-topology (this fact is contained in the proof of Proposition \ref{MV-space is weak induced}), while the latter is a consequence of Proposition \ref{MV-space is weak induced} and \cite[Theorem 2.3]{martin}.

Now, in order to understand the (fuzzy) topological behaviour of Stone-\v Cech compactification for MV-topologies, namely, to determine in which cases an MV-topolog\-ical space homeomorphically embeds into $\widehat \beta X$, it is important to recall that there exist various fuzzy versions of the $T_0$ separation axiom. As stated by the second author in the first work on MV-topologies, one of the main motivations for the introduction of such spaces was to have a well-behaved extension of the category of classical topological spaces, namely, a category of fuzzy topological spaces which contained all classical topologies as a full subcategory and in which the fuzzy versions of the main topological properties could coincide with the original ones for ordinary topologies. Therefore, we believe that the fuzzy $T_0$ axiom which is most suitable for MV-topologies is the one introduced in \cite{sri}:
\begin{definition}\label{t0def}
A fuzzy topological space $\la X, \tau \ra$ is said to be \emph{fuzzy $T_0$} iff for any two distinct points $x, y \in X$, there exists $\alpha \in \tau$ such that either $\alpha(x)= 1$ and $\alpha(y) =0$ or $\alpha(y)= 1$ and $\alpha(x)=0$.
\end{definition}
In \cite{sri}, the authors compared such definition with those presented in \cite{hure} and \cite{puliu}, and they found out that, while the three fuzzy $T_0$ axioms are independent in the case of Chang's fuzzy topologies, the one in their paper implies the other two in the case of Lowen's fuzzy spaces. It is worth remarking also that both the axioms in \cite{hure} and \cite{puliu} imply that no classical topological space is fuzzy $T_0$. Let us also stress out the following important properties, whose proofs are straightforward and will be omitted:
\begin{lemma}\label{t0}
\begin{enumerate}
\item[(i)] An MV-topological space $\la X, \tau\ra$ is fuzzy $T_0$ iff for any two distinct points $x, y \in X$, there exists $\alpha \in \tau$ such that either $\alpha(x) > 0$ and $\alpha(y) =0$ or $\alpha(y) > 0$ and $\alpha(x)=0$.
\item[(ii)] An MV-topological space $\la X, \tau\ra$ is fuzzy $T_0$ iff its initial topology is $T_0$.
\end{enumerate}
\end{lemma}

Last, we recall that a fuzzy topological space is \emph{ultra completely regular} if its initial topology is completely regular \cite{martin2}.

Now remembering once again that the functors $\iota$ and $j$ concide on MV-topologies, i.e., the initial topology of an MV-topology $\la X, \tau\ra$ is simply the family of crisp open sets of $\tau$, we easily get the following characterization.

\begin{theorem}\label{t2mart}
Let $\la X, \tau\ra$ be an MV-topological spaces. The following are equivalent:
\begin{enumerate}
\item[(a)] $X$ is homeomorphically embeddable in $\widehat \beta X$;
\item[(b)] $X$ is fuzzy $T_0$ and ultra completely regular;
\item[(c)] $\iota(X)$ is $T_0$ and completely regular (i.e., is a Tychonoff space);
\item[(d)] $\iota(X)$ is homeomorphically embeddable in $\beta\iota(X)$.
\end{enumerate}
\end{theorem}
\begin{proof}
The equivalence between (b) and (c) follows readily from the definition of ultra complete regularity and Lemma \ref{t0}.

Moreover, it is immediate to see that an MV-topology $\la X, \tau \ra$ is $T_2$ in the sense of Definition \ref{t2ax} iff its initial topology is Hausdorff. Since a Tychonoff space is Hausdorff, also a $T_0$ ultra completely regular MV-topology is $T_2$. Therefore, the equivalence between (a) and (b) is just an immediate application of \cite[Theorem 4.7]{martin2}.

Last, the equivalence between (c) and (d) is a well-known result about classical Stone-\v Cech compactification.
\end{proof}
%\end{proposition}
%\begin{proof}
%First of all, it follows from the remarks right after Definition 3.7 in \cite{rus2}, that the $T_2$ property for fuzzy topologies defined in \cite{puliu} is equivalent to the Hausdorff property (Definition \ref{t2ax}) for all MV-topological spaces. Therefore, by the results of the present section, any Hausdorff MV-topological space $\la X, \tau \ra$ has a $T_2$-fuzzy compactification. Then, by \cite[Theorem 4.7]{martin2}, $\la X, \tau \ra$ is ultra completely regular.
%\end{proof}

\end{document}